\documentclass[a4paper]{amsart}

\usepackage{amssymb, amsthm, amsmath}
\usepackage{mathrsfs}
\usepackage{latexsym}
\usepackage{enumerate}
\usepackage{datetime}
\usepackage[colorlinks=true, linkcolor=blue,pdfstartview=FitV,citecolor=green, urlcolor=blue]{hyperref}

\usepackage[capitalize]{cleveref}

\usepackage{thmtools}
\declaretheorem{theorem}
\declaretheorem[sibling=theorem]{lemma}
\declaretheorem[sibling=theorem]{corollary}
\declaretheorem[sibling=theorem]{proposition}

\declaretheorem[sibling=theorem, style=definition]{definition}
\declaretheorem[sibling=theorem]{question}

\newcommand{\A}{\mathcal{A}}
\newcommand{\B}{\mathcal{B}}
\newcommand{\C}{\mathcal{C}}
\newcommand{\D}{\mathcal{D}}
\newcommand{\G}{\mathcal{G}}
\newcommand{\K}{\mathcal{K}}

\renewcommand{\L}{\mathcal{L}}
\newcommand{\PropL}{\Nat\hbox{-}L^{p}_{\omega_1\omega}}
\newcommand{\PropLc}{\Nat\hbox{-}L^{p}_{c\omega}}

\usepackage{amsthm}
\usepackage{thmtools}
\usepackage{scalerel}
\usepackage{stmaryrd}
\IfFileExists{microtype.sty}{
  \usepackage[]{microtype}
  \UseMicrotypeSet[protrusion]{basicmath} 
}{}
\newcommand\vvee{\hstretch{.8}{\vee\mkern-8mu\vee}}       
\newcommand\wwedge{\hstretch{.8}{\wedge\mkern-8mu\wedge}} 

\newcommand{\pSigma}{\Nat\hbox{-}\Sigma}
\newcommand{\pPi}{\Nat\hbox{-}\Pi}
\newcommand{\bSigma}{\pmb\Sigma}
\newcommand{\bDelta}{\pmb\Delta}
\newcommand{\bPi}{\pmb\Pi}

\newcommand{\pXS}[1][X]{\Sigma^{p,#1}}
\newcommand{\pXP}[1][X]{\Pi^{p,#1}}
\newcommand{\cS}{\Sigma^{pc}}
\newcommand{\cP}{\Pi^{pc}}
\newcommand{\Nat}{\mathbb{N}}

\newcommand{\Ord}{\mathcal{O}}
\newcommand{\Mod}{\text{Mod}}

\newcommand{\forces}{\Vdash_\A}
\newcommand{\defiff}{\mathrel{\mathord{:}\mathord{\iff}}}
\newcommand{\pair}[1]{\langle{#1}\rangle}
\newcommand{\ov}[1]{\overline{#1}}
\newcommand{\code}[1]{\ulcorner{#1}\urcorner}
\newcommand{\posinf}{L^p_{\omega_1\omega}}
\newcommand{\comp}[1]{2^\omega\setminus{#1}}
\renewcommand{\phi}{\xi}
\usepackage{todonotes}
\usepackage[backend=biber,style=alphabetic,url=false]{biblatex}
\author[N.\ Bazhenov]{Nikolay Bazhenov}
\address{Sobolev Institute of Mathematics, Novosibirsk}
\email{bazhenov@math.nsc.ru}

\author[E.\ Fokina]{Ekaterina Fokina}
\address{Institute of Discrete Mathematics and Geometry, Technische Universit\"at Wien}
\email{ekaterina.fokina@tuwien.ac.at}
\author[D.\ Rossegger]{Dino Rossegger}
\address{Department of Mathematics, University of California, Berkeley {\normalfont and} Institute of Discrete Mathematics and Geometry, Technische Universit\"at Wien}
\email{dino@math.berkeley.edu}

\author[A.\ Soskova]{Alexandra Soskova}
\address{
Sofia University, Faculty of Mathematics and Informatics\\ 
}
\email{asoskova@fmi.uni-sofia.bg}
\author[S.\ Vatev]{Stefan Vatev}
\address{
Sofia University, Faculty of Mathematics and Informatics\\ 
}
\email{stefanv@fmi.uni-sofia.bg}

\thanks{Draft compiled at \currenttime\ on \today. The work of the third author was supported by the European Union's Horizon 2020 Research and Innovation Programme under the Marie Sk\l{}odowska-Curie grant agreement No. 101026834 — ACOSE. This project received support
from the Austrian Agency for International Cooperation in Education and Research under grant
WTZ-BG11/2019 and from the Bulgarian National Science Fund through contracts KP-06-Austria-04/06.08.2019.
The fourth and the fifth author were also partially supported by Sofia University Science Fund, project 80-10-134/20.05.2022.}

\title{A Lopez-Escobar theorem for continuous domains}

\begin{document}
\maketitle
\begin{abstract}
  We prove an effective version of the Lopez-Escobar theorem for continuous
  domains. Let $Mod(\tau)$ be the set of countable structures with universe $\omega$ in vocabulary
  $\tau$ topologized by the Scott topology. We show that an invariant set $X\subseteq Mod(\tau)$
  is $\Pi^0_\alpha$ in the effective Borel hierarchy of this topology if and only if it is definable by a $\Pi^p_\alpha$-formula,
  a positive $\Pi^0_\alpha$ formula in the infinitary logic
  $L_{\omega_1\omega}$. As a corollary of this result we obtain a new pullback
  theorem for positive computable embeddings: Let $\K$ be positively computably embeddable in $\K'$ by $\Phi$, then for every $\Pi^p_\alpha$ formula $\phi$ in the vocabulary of $\K'$ there is a $\Pi^p_\alpha$ formula $\phi^*$ in the vocabulary of $\K$ such that for all $\A\in \K$, $\A\models \phi^*$ if and only if $\Phi(\A)\models \phi$. We use this to obtain new results on the possibility of positive computable embeddings into the class of linear orderings.
\end{abstract}
\section{Introduction}\label{sec:introduction}
A common theme in the study of countable mathematical structures is that
isomorphism invariant properties have syntactic characterizations. Examples of
this phenomenon are the existence of Scott sentences for countable
structures~\cite{scott1963}, the Lopez-Escobar theorem, which says that every
invariant Borel subset in the space of countable structures is
definable in the infinitary logic $L_{\omega_1\omega}$~\cite{lopez-escobar1969}, or
that a relation on a structure that is $\Sigma^0_\alpha$ in every copy is definable by a
computable $\Sigma^0_\alpha$ $L_{\omega_1\omega}$-formula~\cite{ash1989,chisholm1990}.

Definability in the logic $L_{\omega_1\omega}$ is closely related to Turing reducibility. The base case of~\cite{ash1989,chisholm1990} tells us that a relation is definable by a computable $\Sigma^0_1$ formula and a computable $\Pi^0_1$ formula if and only if it is computable in every copy. Similarly, evaluating whether a computable infinitary $\Sigma^0_\alpha$ formula holds of a structure $\A$ is c.e.\ in the $\alpha$th Turing jump of $\A$, and hence deciding membership of computable structures in $\Sigma^0_\alpha$ sets is $\Sigma^0_\alpha$. Another fact that makes this connection is that the space of countable $\tau$-structures with universe $\omega$, $Mod(\tau)$, can be viewed as a closed subspace of Cantor space $\mathcal C$. It is well known that continuous functions $\mathcal C\to \mathcal C$ are $X$-computable operators relative to some $X\in 2^\omega$, and thus Turing computable operators correspond to the effectively continuous functions. 

Being a powerful tool to study algorithmic properties, Turing reducibility is not well suited to deal with partial objects, such as c.e.\ sets, partial functions, partial structures, etc. Enumeration reducibility was introduced by Friedberg and Rogers~\cite{FR59} as a generalization of relative computability intended to deal with partial objects. Namely,  $X\subseteq \omega$ is enumeration reducible to $Y\subseteq \omega$, $X\leq_e Y$, if every enumeration of $Y$ computes an enumeration of $X$. 
Given the explained above relationship between definability and Turing computability, one wonders whether a similar relationship can be established between definability and enumeration reducibility.
To establish such a correspondence, throughout the paper we will treat structures as partial objects given by their positive diagrams. For a structure $\A$ in a relational vocabulary $\tau$ its positive diagram is the set $\bigoplus_{R_i\in \tau} R_i^\A$.

Algorithmic properties of structures given by their positive diagrams have been studied for a long time, as this kind of presentation has deep connections to results in algebra and theoretical computer science. The formal general definition of positive (or c.e.)\ structures was given by Mal'cev~\cite{Malcev61}. From an algebraic point of view, many finitely presented algebras are naturally positive but lack a computable presentation~\cite{Markov47, Novikov55, Boone59}. An extensive overview of results on positive (c.e.)\ structures from the 20th century can be found in~\cite{Selivanov03}. Later c.e. structures were studied by Harizanov and co-authors~\cite{CHR11, CHR14,DH17}. In a recent survey~\cite{Khoussainov18} Khoussainov summarizes more results on c.e.\  structures and suggests new questions, relating the study of c.e.\ structures to the theory of equivalence relations.

The approach that applies enumeration reducibility to study computational properties of structures goes back to Soskov~\cite{soskov2004,soskov-baleva-2006} who adapted notions studied in computable structure theory to enumeration reducibility. Assuming that a structure  is given by its positive diagram, one may ask questions similar to the ones in classic computable structure theory but with respect to enumeration reducibility. For example, Richter introduced the notion of the Turing degree spectrum of a structure, the set of Turing degrees of its isomorphic copies~\cite{Richter81}. The  first and one of the fundamental results in this area of Knight is a dichotomy~\cite{knight1986}. If a structure is automorphically non-trivial, its degree spectrum is upwards closed. Otherwise, it is a singleton. This dichotomy disappears in the analogue of this notion for enumeration reducibility---enumeration degree spectra. The enumeration degree spectrum of a structure is the set of degrees of enumerations of its positive diagram~\cite{soskov2004}. Soskov's work on this notion exposed deep theory and contributed to the study of Turing degree spectra, see~\cite{soskova2017} for a survey. In~\cite{csima2022} r.i.p.e.\ relations, relations that are enumerable from the positive diagrams of all isomorphic copies of a given structure, were studied. The main results show that a relation is r.i.p.e.\ if and only if it is definable by a positive computable $\Sigma^0_1$ infinitary formula, and that positive enumerable functors---a reducibility between structures introduced in~\cite{csima2021}---have a syntactic equivalent in a version of interpretability using $\Sigma^p_1$ formulas.

In this article, we establish another relationship between definability by positive infinitary formulas and enumeration reducibility. Our main result is an effective version of the Lopez-Escobar theorem that establishes a connection between the Borel hierarchy on the Scott topology on the space of countable structures and a hierarchy of positive $L_{\omega_1\omega}$-formulas.
Let $(Mod(\tau), \preceq)$ be the set of countable structures in the relational language $\tau$ with universe $\omega$, where $\preceq$ is the substructure ordering on $Mod(\tau)$. This ordering gives rise to a topology, the Scott topology with basic open sets the superstructures of $\tau$-finite structures. It is not a Polish space but an $\omega$-continuous domain and allows the definition of a Borel hierarchy. Selivanov~\cite{selivanov2006} showed that the Borel sets in this topology equal the Borel sets under the usual Polish topology on $Mod(\tau)$ and, thus, one gets Lopez-Escobar's theorem for the Scott topology for free. However, the Borel hierarchies on this spaces differ and we will prove a stronger version of this theorem, matching the Borel complexity with the complexity of defining formulas. We show that a set is $\Pi^0_\alpha$ in the effective Borel hierarchy on the Scott topology if and only if it is definable by a computable $\Pi^p_\alpha$ formula (\cref{thm:efflopezescobar}). The $p$ in $\Pi^p_\alpha$ stands for positive; the $\Pi^p_\alpha$ formulas are a subset of the classical $\Pi^0_\alpha$ formulas in the logic $L_{\omega_1\omega}$ with the use of negations restricted. Vanden Boom~\cite{vandenboom2007} showed the analogue of our theorem in the classical setting. The bold-face version is a corollary of Vaught's proof of the Lopez-Escobar theorem~\cite{vaught}. We will prove our version of the Lopez-Escobar theorem using a notion of forcing which can be seen as an amalgamation of the forcing used in Vanden Boom's proof and a forcing used by Soskov~\cite{soskov2004}.

As Turing computable operators are to Polish spaces, enumeration operators are to the Scott topology. Enumeration operators can be viewed as functions mapping subsets of $\omega$ to subsets of $\omega$ and thus as functions from $2^\omega$ to $2^\omega$. While we have only defined the Scott topology on $Mod(\tau)$, there is a similar topology on $2^\omega$ where the basic open sets are the supersets of finite sets. Enumeration operators correspond to effective Scott continuous functions from $2^\omega$ to $2^\omega$. As in the classical case, we can view the Scott topology on $Mod(\tau)$ as a subspace of the Scott topology on $2^{\omega}$. We use this fact to obtain a new pullback theorem for computable embeddings. A computable embedding between two classes $\K$ and $\K'$ is an enumeration operator $\Phi$ such that for $\A\in \K$, $\Phi(\A)\in K'$ and $\A\cong \B$ if and only if $\Phi(\A)\cong \Phi(\B)$~\cite{calvert2004}. Our pullback theorem says that for $\phi$ a $\Pi^p_\alpha$ formula in the vocabulary of $\K'$, there is $\phi^*$, a $\Pi^p_\alpha$ formula in the vocabulary of $\K$ such that for all $\A\in \K$, $\A\models \phi^*$ if and only if $\Phi(\A)\models \phi$ (\cref{thm:pullback}).

Computable embeddings and their classical analogue, Turing computable embeddings~\cite{knight2007}, have been heavily studied in computable structure theory over the last decade. One question that has seen interest is to find a dividing line between these two types of reducibilities~\cite{calvert2004,chisholm2007,kalimullin2012,kalimullin_computable_2018,bazhenov_computable_2021}: What properties must classes of structures $\K$ and $\K'$ have so that there is a Turing computable embedding from $\K$ to $\K'$ but no computable embedding? Our pullback theorem is a useful tool for investigations of this kind. To demonstrate this we study classes $\K$ and their sister classes $\tilde \K$ where $\tilde\A$ is obtained from $\A$ by introducing an infinite equivalence relation $E$ on $\omega$ where all equivalence classes are infinite such that $\tilde\A/E\cong \A$. We continue work of Ganchev, Kalimullin, and Vatev~\cite{kalimullin2018} and study when we can have a computable embedding from $\tilde\K$ into $\K$. We show that for $\L$ the class of linear orderings, any two structures in $\tilde\L$ satisfy the same $\Sigma_2^p$ sentences. Thus, if $\K\subseteq \L$, and the orderings in $\K$ are distinguishable by $\Sigma^p_2$ sentences, we can not have a computable embedding $\tilde\K \to \K$. Using this we show that there is no embedding from $\{\tilde\omega,\tilde{\omega}^*\}$ to $\{\omega,\omega^*\}$, answering a question of Kalimullin (\cref{thm:omegaomegastar}).

\section{Preliminaries}\label{sec:preliminaries}
\subsection{Scott topology}
Let $(P,\leq)$ be a partial order. A set $D\subseteq P$ is \emph{directed} if
$D\neq \emptyset$ and for all $a,b\in D$, there is $d\in D$ with $a\leq d$ and
$b\leq d$. A partial order $P$ is a \emph{directed complete partial order
(dcpo)} if every directed $D\subseteq P$ has a supremum $\sup D$ in $P$.
A set $U$ is \emph{Scott open} if $U$ is an upper set, i.e., $x\in U$ and
$x\leq y\implies y\in U$, and for every directed set $D$ with $\sup D\in U$ we have
that $D\cap U\neq \emptyset$. The Scott open sets of $P$ form a topology on
$P$, the \emph{Scott topology}.

Notice that the Scott topology on any partial ordered set that contains a chain
of size larger than $2$ is not Hausdorff. To see this just note that for all
$x,y\in P$ if $x\leq y$ then every Scott open set containing $x$ contains
$y$. So, in particular, it is not Polish. 

As an example of such a space consider the Scott topology on $2^\omega$
equipped with the dcpo given by $f\subseteq g$ if $f(i)=1 \implies g(i)=1$. It
has a natural countable basis given by the basic open sets
\[2^\omega, \emptyset, \text{ and } O_n = \{f \in 2^\omega \mid f(n) = 1\}
\text{ for all }n\in\omega.\]
The Scott topology on $2^\omega$ is what is commonly referred to as an
$\omega$-continuous domain. In fact it is an $\omega$-algebraic domain and thus
a quasi-Polish space~\cite{debrecht2013}. The spaces we consider in
this paper will be closed subspaces of the Scott topology on $2^\omega$.

Let $\tau$ be a countable relational vocabulary containing $=$ and $\neq$. Denote by $Mod(\tau)$ the set of
countable $\tau$-structures with universe $\omega$. Similarly, if $\phi$ is a $\tau$-formula,
denote by $Mod(\phi)$ the set of models of $\phi$. There is a natural dcpo on $Mod(\tau)$ given by
\[\A\preceq \B \defiff \forall R\in \tau (R^\A\subseteq R^\B).\]
For sake of readibility we will refer to the Scott topology on
$(Mod(\tau),\preceq)$ simply as the Scott topology on $Mod(\tau)$. It might
seem odd to the reader that we include both $=$ and $\neq$ in the vocabulary.
The reason for this is that we want to identify the Scott open subsets of
$Mod(\tau)$ with subsets that are definable by existential formulas without negation. If $\neq$
is not in the vocabulary, then this is generally not possible for the emptyset. With $\neq$
in the vocabulary we can always define the emptyset by the formula $\exists
x\; x\neq x$ regardless of other properties of $\tau$. 

  In order to use tools from computability theory, it is convenient to represent
$\A\in Mod(\tau )$ by
their atomic\footnote{When talking about computable structures, there is often no algorithmic difference between atomic and quantifier-free formulas. This often leads to identifying atomic and quantifier-free diagrams, the two terms being then used interchangably. Here we really do mean atomic formulas, in particular, no negations are allowed.} diagram -- the infinite binary string 
\[ D_\A(n)=\begin{cases} 1 & \A \models \phi^{at}_n [x_i\to i] \\
0 & \text{otherwise}\end{cases}\]
where $\phi^{at}_n$ is the $n^{th}$ formula in an effective enumeration of all
atomic  $\tau$-formulas with free variables a subset of $\{x_0,\dots,x_n\}$.
Notice that this representation gives a homeomorphism between
$(Mod(\tau),\preceq)$ and the closed subset of $(2^{\omega},\subseteq)$ given
by
\[\bigcap_{n\in \{\code{x_i=x_j}:i\neq j\in\omega\}} \comp{O_{n}}\ \cap\ \bigcap_{n\in\{\code{x_i\neq x_i}:i\in\omega\}} \comp{O_n}.\]
Notice that the Scott topology on $Mod(\tau)$
is homeomorphic to the subspace topology on this space. By the same idea we can view natural classes of structures such as linear orderings as closed subspaces of $2^\omega$. 
%
\subsection{The Borel hierarchy for non-metrizable spaces} 
The Borel sets of a Polish space bear a natural structures in terms of the
Borel hierarchy. This is witnessed by the fact that the $\bSigma^0_2$ sets
are precisely the $F_\sigma$ subsets and the $\bPi^0_2$ sets coincide with the
$G_\delta$ subsets of the space. In non-Hausdorff spaces we can not find this
correspondence in general. It is common for non-Hausdorff spaces to have open sets that are not $F_\sigma$ (i.e., countable unions of closed sets)
and closed sets that are not $G_\delta$ (i.e., countable intersections of open sets).
One example of this phenomenon is the Sierpi{\'n}ski space, which has $\{\bot,\top\}$ as underlying set and the singleton $\{\top\}$ open but not closed.
This implies that the classical definition of the Borel hierarchy, which defines level $\bSigma^0_2$ as the $F_\sigma$-sets
and $\bPi^0_2$ as the $G_\delta$-sets, is not appropriate in the general setting.
\begin{definition}[\cite{selivanov2006}]
  Let $(X,\tau)$ be a topological space. For each countable ordinal $\alpha \geq 1$ we define $\bSigma^0_\alpha(X,\tau)$ inductively as follows.
  \begin{enumerate}
  \item
    $\bSigma^0_1(X,\tau) = \tau$.
  \item
    For $\alpha > 1$, $\bSigma^0_\alpha(X,\tau)$ is the set of all subsets $A$ of $X$ which can be expressed in the form
    \[A = \bigcup_{i\in\omega}B_i \setminus B'_i,\]
    where for each $i$, $B_i$ and $B'_i$ are in $\bSigma^0_{\beta_i}(X,\tau)$ for some $\beta_i < \alpha$.
  \end{enumerate}
  We define $\bPi^0_\alpha(X,\tau) = \{X \setminus A \mid A \in \bSigma^0_\alpha(X,\tau)\}$
  and $\bDelta^0_\alpha(X,\tau) = \bSigma^0_\alpha(X,\tau) \cap \bPi^0_\alpha(X,\tau)$.
  Finally, we define $\mathbf{B}(X,\tau) = \bigcup_{\alpha<\omega_1} \bSigma^0_\alpha(X,\tau)$ to be the \emph{Borel} subsets of $(X,\tau)$.
\end{definition}
The definition above is equivalent to the classical definition of the Borel hierarchy on metrizable spaces, but differs in general.
Selivanov has investigated this hierarchy in a series of papers (\cite{selivanov2005,selivanov2006}), with an emphasis on applications to $\omega$-continuous
domains. He shows (Proposition 5.3 in \cite{selivanov2005}) that the classical
Borel hierarchy on Cantor space differs from the new Borel hierarchy on the Scott space at finite levels
and coincides at infinite levels. 
\subsection{A hierarchy of positive infinitary formulas}
In the classical setting the Lopez-Escobar theorem establishes a correspondence
between subsets of $Mod(\tau)$ defined by sentences in the infinitary logic
$L_{\omega_1\omega}$ and the Borel sets. Our theorem relates the Borel sets in
the Scott topology with subsets defined by positive infinitary sentences
defined below. See Ash and Knight~\cite{ash2000} for an indepth treatment of infinitary logic
that goes in hand with the developments in this section.
\begin{definition}\label{def:posinf}
  Fix a countable vocabulary $\tau$. For every $\alpha<\omega_1$ define the sets of $\Sigma^p_\alpha$ and
  $\Pi^p_\alpha$ $\tau$-formulas inductively as follows.
\begin{itemize}
\item
  Let $\alpha = 0$. Then:
  \begin{itemize}
  \item
    the $\Sigma^p_0$ formulas are the finite conjunctions of atomic $\tau$-formulas.
  \item
    the $\Pi^p_0$ formulas are the finite disjunctions of \emph{negations} of
    atomic $\tau$-formulas.
  \end{itemize}
\item
  Let $\alpha = 1$. Then:
  \begin{itemize}
  \item
    $\varphi(\bar{u})$ is a $\Sigma^p_1$ formula if it has the form
    \[\varphi(\bar{u}) = \vvee_{i\in I} \exists \bar{x}_i \psi_i(\bar{u},\bar{x}_i), \]
    where for each $i \in I$, $\psi_i(\bar{u},\bar{x}_i)$ is a $\Sigma^p_0$ formula, $I$ is countable.
  \item
    $\varphi(\bar{u})$ is a $\Pi^p_1$ formula if it has the form
    \[\varphi(\bar{u}) = \wwedge_{i\in I} \forall \ov{x}_i \psi_i(\bar{u},\bar{x}_i), \]
    where for each $i \in I$, $\psi_i(\bar{u},\bar{x}_i)$ is a $\Pi^p_0$ formula, $I$ is countable.    
  \end{itemize}
\item
  Let $\alpha \geq 2$. Then:
  \begin{itemize}
  \item
    $\varphi(\bar{u})$ is $\Sigma^p_{\alpha}$ formula if it has the form
    \[\varphi(\bar{u}) = \vvee_{i\in I} \exists \bar{x}_i(\phi_i(\bar{u},\bar{x}_i) \land \psi_i(\bar{u},\bar{x}_i)), \]
    where for each $i \in I$, $\phi_i(\bar{u},\bar{x}_i)$ is a $\Sigma^p_{\beta_i}$ formula and $\psi_i(\bar{u},\bar{x}_i)$ is $\Pi^p_{\beta_i}$ a formula, for some $\beta_i < \alpha$ and $I$ countable.
  \item
    $\varphi(\bar{u})$ is $\Pi^p_{\alpha}$ formula if it has the form
    \[\varphi(\bar{u}) = \wwedge_{i\in I} \forall \bar{x}_i(\phi_i(\bar{u},\bar{x}_i) \lor \psi_i(\bar{u},\bar{x}_i)), \]
    where for each $i \in I$, $\phi_i(\bar{u},\bar{x}_i)$ is a $\Sigma^p_{\beta_i}$ formula and $\psi_i(\bar{u},\bar{x}_i)$ is a $\Pi^p_{\beta_i}$ formula, for some $\beta_i < \alpha$ and $I$ countable.
  \end{itemize}
\end{itemize}
The set of positive infinitary formulas, $\posinf$, is the smallest subset of
$L_{\omega_1\omega}$ closed under logical equivalence containing all $\Pi^p_\alpha$ formulas for
$\alpha<\omega_1$.
\end{definition}
It is not hard to see that every $L_{\omega_1\omega}$ formula is
logically equivalent to a $\posinf$-formula and that these two sets are thus equivalent. However, as should be evident from the definition, the hierarchies
differ at finite levels.

We say that a  formula is in \emph{normal form} if it is $\Sigma^p_\alpha$ or
$\Pi^p_\alpha$ for some $\alpha<\omega_1$. An easy transfinite induction shows
that every $L_{\omega_1\omega}$ formula is equivalent to a formula in a normal form.

Consider a $\Sigma^p_1$ formula $\phi=\vvee_{i\in I}\exists \bar x_i \psi_i(\bar
x_i)$ in a computable vocabulary $\tau$. We can fix an injective computable enumeration of the
conjunctions of atomic $\tau$-formulas and assume without loss of generality that $i=\code{\psi_i}$ is the index
of $\psi_i$ in this enumeration. Let $X\subseteq \omega$ such that $I$ is
$X$-c.e., then we say that $\phi$ is an $X$-computable $\Sigma^p_1$ formula, or
short, $\phi\in \pXS_1$. We will denote the set of $\Sigma^p_1$
formulas such that $I$ is c.e., as the $\cS_1$ formulas, where $c$ stands
for computable.

We could intuitively generalize the above definition to arbitrary $X$-computable
$\alpha$ and $X$-computable $\Pi^p_\alpha$ formulas by demanding that all the index sets $I$
are $X$-c.e. However, doing this formally requires some care. Let us give the
formal definition for computable $\Sigma^p_\alpha$ and $\Pi^p_\alpha$. It
should be obvious that this definition relativizes.

\begin{definition}\label{def:compposinf}
For $a \in \mathcal{O}$, we define the index sets $S^\Sigma_a$ and $S^\Pi_a$ as follows:
\begin{itemize}
\item
  Let $|a|_{\Ord} = 0$ and
    $S^\Sigma_1 = \{\code{\varphi} \mid \varphi\text{ a finitary conjunction
    of atomic $\tau$-formulas}\}$
\item
  Let $|a|_{\mathcal{O}} > 0$. Then:
  \begin{itemize}
  \item
    $S^\Sigma_{a} = \{\pair{\Sigma,a,\overline{x},e} \mid \overline{x} \text{ is a tuple of variables, } e \in \omega\}$,
  \item
    $S^\Pi_{a} = \{\pair{\Pi,a,\overline{x},e} \mid \overline{x} \text{ is a tuple of variables, } e \in \omega\}$.
  \end{itemize}
\end{itemize}

We define the positive infinitary formula $\varphi_i(\overline{u})$ with index $i$ as follows.

\begin{itemize}
\item
  Let $|a|_{\Ord} = 0$.
  If $i \in S^\Sigma_1$, 
  then $\varphi_i(\overline{u})$ is
  the $\tau$-formula with G\"odel number $i$.
\item
  Let $|a|_{\mathcal{O}} = 1$. Then:
  \begin{itemize}
  \item
    If $i \in S^\Sigma_a$, then $i = \pair{\Sigma,a,\ov{u},e}$, for some tuple $\ov{u}$ and index $e$, and
    $\varphi_i(\overline{u})$ is the disjunction of formulas $(\exists\overline{x})\phi_j(\overline{u},\overline{x})$, where $j \in W_e \cap S^\Sigma_1$.
  \item
    If $i \in S^\Pi_a$, then $i = \pair{\Pi,a,\ov{u},e}$, for some tuple $\ov{u}$ and index $e$, and
    $\varphi_i(\overline{u})$ is the conjunction of formulas
    $(\forall\overline{x})\neg \phi_j(\overline{u},\overline{x})$, where $j \in
    W_e \cap S^\Sigma_1$.
  \end{itemize}
\item
  Let $|a|_{\mathcal{O}} \geq 2$. Then:
  \begin{itemize}
  \item
    If $i \in S^\Sigma_a$, then
    $i = \pair{\Sigma,a,\ov{u},e}$, for some tuple $\ov{u}$ and index $e$, and $\varphi_i(\ov{u})$ is the
    disjunction of formulas $(\exists\ov{x})[\phi_\ell(\ov{u},\ov{x}) \land \psi_r(\ov{u},\ov{x})]$,
    where $\pair{\ell,r} \in W_e$ and $\ell \in S^\Sigma_b$ and $r \in S^\Pi_b$, for some $b <_{\Ord} a$ with $|b|_{\mathcal{O}} > 0$.
    Moreover, the third components of the indices $\ell$ and $r$ is the tuple $(\ov{u},\ov{x})$.
  \item
    If $i \in S^\Pi_a$, then
    $i = \pair{\Pi,a,\ov{u},e}$, for some tuple $\ov{u}$ and index $e$, and $\varphi_i(\ov{u})$ is the
    conjunction of formulas $(\forall\ov{x})[\phi_\ell(\ov{u},\ov{x}) \lor \psi_r(\ov{u},\ov{x})]$,
    where $\pair{\ell,r} \in W_e$ and $\ell \in S^\Sigma_b$ and $r \in S^\Pi_b$, for some $b <_{\Ord} a$ with $|b|_{\mathcal{O}} > 0$.
    Moreover, the third components of the indices $\ell$ and $r$ is the tuple $(\ov{u},\ov{x})$.
  \end{itemize}
\end{itemize}
If $i\in S_a^\Sigma$ ($S_a^\Pi$) with $|a|=\alpha$, then $\varphi_i\in \cS_\alpha$
($\cP_\alpha$). A formula $\phi$ is a \emph{computable positive infinitary
formula} if $\phi\in L_{c\omega}^p=\bigcup_{\alpha<\omega_1^{\mathrm CK}}
  \Sigma_\alpha^{pc}\cup \Pi_\alpha^{pc}$.
\end{definition}
\cref{def:compposinf} can be relativized to arbitary $X\subseteq\omega$
by replacing occurences of $\mathcal O$ and
$W_e$ with $\mathcal O^X$ and $W_e^X$. This allows us to obtain $\pXS_\alpha$
and $\pXP_\alpha$ formulas. As every $L^p_{\omega_1\omega}$ formula $\varphi$ is
$\Pi^p_\alpha$ for some countable ordinal $\alpha$, $\varphi$ is $\pXP_\alpha$ for $X$ the computable join of
the index sets in the definition of $\varphi$ and a set $Y$ with $\alpha <
\omega_1^Y$. Hence, a formula is $\Pi^p_\alpha$ if and only if it is
$\pXP_\alpha$ for some $X\subseteq\omega$.

\begin{definition}\label{def:posinfnegation}
  We define $neg(\varphi)$ following the definition of $\varphi$.
  \begin{itemize}
  \item
    Let $\alpha = 0$.
    \begin{itemize}
    \item
      If $\varphi$ is $\Sigma^p_0$, i.e., $\varphi$ is a conjunction over some finite set $I$ of the atomic formulas $\varphi^{at}_i$, where $i \in I$,
      then $neg(\varphi)$ is the finite disjunction over the same finite set $I$ of $\neg \varphi^{at}_i$, where $i \in I$.
    \item
      If $\varphi$ is $\Pi^p_0$, i.e., $\varphi$ is a disjunction over some finite set $I$ of $\neg \varphi^{at}_i$, where $i \in I$,
      then $neg(\varphi)$ is the conjunction over the same finite set $I$ of $\varphi^{at}_i$, where $i \in I$.
    \end{itemize}
  \item
    Let $\alpha = 1$.
    \begin{itemize}
    \item
      If $\varphi$ is $\Sigma^p_1$, i.e., $\varphi$ is an infinitary disjunction over a set of formulas of the form $\exists \bar{x} \psi$,
      then $neg(\varphi)$ is an infinitary conjunction over the same set of the formulas $\forall \bar{x} neg(\psi)$.
    \item
      If $\varphi$ is $\Pi^p_1$, i.e.,
      $\varphi$ is an infinitary conjunction over a set of formulas of the form $\forall \bar{x} \psi$,
      then $neg(\varphi)$ is an infinitary disjunction over the same set of the formulas $\exists \bar{x} neg(\psi)$.
    \end{itemize}
  \item
    Let $\alpha \geq 2$.
    \begin{itemize}
    \item
      If $\varphi$ is $\Sigma^p_{\alpha}$ formula, i.e.,
      $\varphi$ is an infinitary disjunction over a set of formulas of the form $\exists \bar{x}(\phi \land \psi)$,
      then $neg(\varphi)$ is an infinitary conjunction over the same set of the formulas $\forall \bar{x}(neg(\phi) \lor neg(\psi))$.
    \item
      If $\varphi$ is $\Pi^p_{\alpha}$ formula, i.e.,
      $\varphi$ is an infinitary conjunction over a set of formulas of the form $\forall \bar{x}(\phi \lor \psi)$,
      then $neg(\varphi)$ is an infinitary disjunction 
      over the same set of the formulas $\exists \bar{x}(neg(\phi) \land neg(\psi))$.
    \end{itemize}
  \end{itemize}
  If $\phi$ is not in normal form and thus not in $\Sigma^p_\alpha$ or $\Pi^p_\alpha$ for any $\alpha<\omega_1$, then we can find a formula in normal form $\psi$ such that $\psi\equiv \phi$ and let $neg(\phi)=neg(\psi)$.
\end{definition}

The above definition makes the proof of the following proposition self-evident.
\begin{proposition}
  For any $L^p_{\omega_1\omega}$ formula $\varphi$, $neg(neg(\varphi)) \equiv \varphi$.
  Furthermore, if $\varphi\in \cP_\alpha$, then $neg(\varphi)\in \cS_\alpha$
  and $neg(neg(\varphi))=\varphi$.
\end{proposition}

Following \cite[Chapter V.4]{montalban2}, consider $\posinf$ formulas in the language of arithmetic with one extra unary relation symbol $D$, to represent the diagram of the structure, which we treat as a second-order variable.
This means that we allow atomic formulas of the form $D(x)$.
We will refer to these formulas as $\PropL$ formulas.
We can now talk about $\PropL$ definable subsets of $2^{\omega}$ (or $\Mod(\tau)$)
in the standard model of arithmetic.
A set of natural numbers $A$ is \emph{$\PropL$-definable}
if there is a $\PropL$ formula $\phi$ such that
$(\Nat,A)\models \phi$ where $\Nat$ is the standard model of arithmetic.
By identifying elements of $2^{\omega}$ and $\Mod(\tau)$ as
characteristic functions of sets we get the notion of definability for these elements.
Furthermore, it is not hard to see that
within this model every $\PropL$ formula is equivalent to
an $\PropL$ formula without quantifiers,
and that this quantifier elimination can be done without
increasing the complexity of the formula.
To see this, simply note that we can
replace formulas of the form $\exists x \phi(x)$ by $\vvee_{n\in\omega} \phi(\mathbf n)$
and formulas of the form $\forall x \phi(x)$ by $\wwedge_{n\in\omega}
\phi(\mathbf n)$ where $\mathbf n$ is the formal term for $n$ in $\Nat$. Notice
that the atomic subformulas of a quantifier-free $\PropL$ formula are either
$D(\mathbf n)$, $\top$, or $\bot$. We will refer to $\Sigma^p_\alpha$ formulas
of this form as $\pSigma^p_\alpha$ formulas and to $\Pi^p_\alpha$ formulas as
$\pPi^p_\alpha$ formulas.

%
\begin{proposition}\label{prop:borelandN}
  For all non-zero $\alpha\leq\omega_1$,
  \begin{enumerate}
    \item $\bSigma^0_\alpha(2^\omega)=Mod(\pSigma^p_\alpha)=\left\{\{ X\in 2^\omega: (\mathbb
  N,X)\models \phi\}:\phi\in \pSigma^p_\alpha\right\}$
    \item $\bPi^0_\alpha(2^\omega)=Mod(\pPi^p_\alpha)=\left\{\{ X\in 2^\omega: (\mathbb
  N,X)\models \phi\}:\phi\in \pPi^p_\alpha\right\}$
  \end{enumerate} \end{proposition}
\begin{proof} The proof is by induction on $\alpha$. We may assume that $\PropL$ formulas
  are quantifier free and all atomic subformulas are of the form $\top$, $\bot$
  or $D(\mathbf n)$ for some $n\in\omega$.
  If $\phi\in \pSigma^p_{0}$ contains the subformula $\bot$, then
  $Mod(\phi)=\emptyset$ and if $\phi=\top$, then $Mod(\phi)=2^\omega$.
  Otherwise, the corresponding open set $X_\phi$ is $\bigcap_{n\in\omega: D(\mathbf
  n)\in sf(\phi)} O_n$, where by $sf(\phi)$ we denote the subformulas of $\phi$.

  If $\phi\in \pSigma^p_1$, i.e., $\phi=\vvee_{i\in I}
  \psi_i$ with $\psi_i\in \pSigma_0^p$, then $Mod(\phi)=X_\phi=\bigcup_{i\in I}
  X_{\psi_i}\in \bSigma^0_1$. If $\phi\in \pPi_1^p$, then
  $\phi=neg(\psi)$ and $Mod(\phi)=Mod(neg(\psi))=\comp{X_\psi}\in \bPi^0_1$.
  On the other hand, if
  $B\in \bSigma^0_1$, i.e., $B=\bigcup_{i\in I}\bigcap_{n\leq N_i} O_{n}$, then
  let $\phi_B=\vvee_{i\in I} \wwedge_{n\leq N_i} D(\mathbf{n})$. If $B\in
  \bPi^0_1$, then $\comp{B}\in \bSigma^0_1$, $neg(\phi_{\comp{B}})\in
  \pPi^{p}_1$ and $B=Mod(neg(\phi_{\comp{B}}))$.

  Assume that for all $\beta<\alpha$ and all $\pSigma_\beta^p$ formulas $\phi$,
  there is a $\bSigma^0_\beta$ set $X_\phi$ such that $Mod(\phi)=X_\phi$.
  Let $\phi\in \pSigma^p_\alpha$, i.e., $\phi=\vvee_{i\in I} (\psi_i\land
  neg(\theta_i))$ for $\psi_i,\theta_i\in \pSigma^p_{\beta_i}$ with
  $\beta_i<\alpha$. Then $Mod(\phi)=\bigcup_{i\in I} X_{\psi_i}\setminus
X_{\theta_i}=X_\phi\in \bSigma^0_\alpha$. If $\phi\in \pPi^p_\alpha$, then
$\phi=neg(\psi)$ for $\psi\in \Sigma^p_\alpha$ and
$Mod(\phi)=Mod(neg(\psi))=\comp{X_\psi}\in \bPi^0_\alpha$.

  Suppose that for all $\beta<\alpha$ and all $\bSigma^0_\beta$ sets $B$,
  there is a $\pSigma^{p}_\beta$ formula $\phi_B$ such that $Mod(\phi_B)=B$.
  Assume that $B\in \bSigma^0_\alpha$, i.e., $B=\bigcup_{i\in I} B_i\setminus
  B_i'$ with $B_i,B_i'\in \bSigma^0_{\beta_i}$ for $\beta_i<\alpha$. Then
  $\phi_B=\vvee_{i\in I} (\phi_{B_i}\land neg(\phi_{B_i'}))\in \pSigma^p_\alpha$ and
  $Mod(\phi_B)=B$. At last, let $B\in \bPi^0_\alpha$, then $\comp{B}\in
  \bSigma^0_\alpha$, $neg(\phi_{\comp{B}})\in \pPi^p_\alpha$ and
  $Mod(neg(\phi_{\comp{B}}))=B$. 
  \end{proof}
\subsection{The effective Borel hierarchy}
As we have discussed after \cref{def:compposinf}, every positive infinitary
formula is computable relative to some oracle $X$.
\cref{prop:borelandN} associates to every Borel set a positive infinitary formula
$\phi$ in the language of arithmetic with an additional unary relation symbol. This motivates
the following definition of the effectively Borel hierarchy on $2^\omega$ and
through this on $Mod(\tau)$ for a given computable vocabulary $\tau$.
\begin{definition}
  Let $X\subseteq 2^\omega$ and $Y\subseteq \omega$.
  \begin{itemize}
    \item $X\in\Sigma^0_\alpha(Y)$ if there is
      a $\pXS[Y]_\alpha$ $\PropL$-formula $\varphi$ such that $X=\{D:(\mathbb N,D)\models
      \varphi\}$. A $\pXS[Y]_\alpha$ index for $\varphi$ is
      a $Y$-computable index for $X$.
\item $X\in\Pi^0_\alpha(Y)$ if there is 
      a $\pXP[Y]_\alpha$ $\PropL$-formula $\varphi$ such that $X=\{D:(\mathbb
      N,D)\models \varphi\}$. A $\pXP[Y]_\alpha$ index for $\varphi$ is
      a $Y$-computable index for $X$.

  \end{itemize}
  Furthermore, $X$ is \emph{effectively $\bSigma^0_\alpha$}, or $X\in
  \Sigma^0_\alpha$, if $X\in \Sigma^0_\alpha(\emptyset)$. Likewise, $X$ is
  \emph{effectively $\bPi^0_\alpha$}, or $X\in \Pi^0_\alpha$, if $X\in
  \Pi^0_\alpha(\emptyset)$. 
\end{definition}

\section{The Lopez-Escobar theorem for positive infinitary formulas}\label{sec:lopez-escobar}
The main goal of this section is to give a forcing proof of the Lopez-Escobar
theorem and its related corollaries. The definition of our forcing relation is
similar to the forcing relation in~\cite{soskov2004} but our presentation of
the forcing and the forcing poset are closer to~\cite{montalban2}.

Our forcing poset is the space of presentations of a fixed structure $\A$.
We produce a generic element $\mathcal G\cong \A$ that decides membership of
the Borel sets on $\A$. We represent the Borel sets by their 
$\PropL$ formulas using the correspondence established
in~\cref{prop:borelandN}. Formally our forcing poset $\mathcal P$ is the set of
finite permutations of the natural numbers. We will build a generic enumeration $g$ and let our generic
proposition $\mathcal G=g^{-1}(\A)$. Let us start by defining our forcing
relation.
\begin{definition}[Forcing relation]\label{def:forcingrelation}
  Fix a countable structure $\A$ in a countable vocabulary and let $\mathcal P$
  be the set of all finite permutations of $A$,  which by our assumption is
$\omega$. For a forcing condition $p\in
  \mathcal P$ define the forcing relation
  $\forces$ inductively for all $\PropL$-formulas $\varphi$ in normal form as follows.
\begin{enumerate}
  \item[(atom1)] $p \forces \top$ and $p \not\forces \bot$
  \item[(atom2)] $p \forces D(n)$ if $\A \models \phi^{at}_n[x_i \mapsto p_i]$
  \item[($\pSigma^p_0$)] if $\varphi=D(n_0) \land \cdots \land D(n_{k})$,
  then $p \Vdash_\A \varphi$ if $p \Vdash_\A D(n_i)$ for all $i \leq k$
\item[($\pSigma^p_1$)]
  if $\varphi= \vvee_{i\in I} \phi_i$, where $\phi_i$ are $\pSigma^p_0$ formulas, then 
  $p \forces \varphi$ if for some $i \in I$, $p \forces \phi_i$
\item[($\pSigma^p_{>1}$)]
  if $\varphi= \vvee_{i\in I} (\phi_i \land \theta_i)$, where for each $i \in
  I$, $\phi_i\in\pSigma^p_{\beta_i}$ and $\theta_i\in \pPi^p_{\beta_i}$ for some $\beta_i>0$, then $p \forces \varphi$ if
  $p \forces \phi_i$ and $p \forces \theta_i$ for some $i \in I$
\item[($\pPi^p_{\alpha}$)]\label{it:piforcingdef}
  if $\varphi$ is a $\pPi^p_\alpha$ formula, for any $\alpha \geq 0$, then
  $p \forces \varphi$ if for all $q \supseteq p$, $q \not\forces neg(\varphi)$
\end{enumerate}
\end{definition}
There is a subtle point to the definition of (atom2). If $\phi_n^{at}$ contains
a variable symbol $x_i$ with $i\geq |p|$, then we regard
$\phi_n^{at}[x_i\mapsto p_i]$ as false in $\A$ and thus $p\not\forces D(n)$ in
this case.

\begin{lemma}[Extension]\label{lem:extension}
  For any $\PropL$ formula $\varphi$ and forcing condition $p$,
  if $p \forces \varphi$ and $p \subseteq q$, then $q \forces \varphi$.
\end{lemma}
\begin{proof}
  Straightforward induction on $\varphi$.
\end{proof}
  



\begin{lemma}[Consistency]\label{lem:consistency}
  For any $\PropL$ formula $\varphi$ and any forcing condition $p$,
  it is not the case that $p \forces \varphi$ and $p \forces neg(\varphi)$.
\end{lemma}
\begin{proof}
  Towards a contradiction, assume that for some $\PropL$ formula $\varphi$ and some forcing condition $p$,
  $p \forces \varphi$ and $p \forces neg(\varphi)$.
  Without loss of generality, suppose that $\varphi$ is a $\pSigma^p_\alpha$ formula, for some $\alpha$.
  Then, clearly, $neg(\varphi)$ is a $\pPi^p_\alpha$ formula.
  It follows that, since $p \forces neg(\varphi)$, for all $q \supseteq p$,  $q \not\forces neg(neg(\varphi))$
  and hence $p \not\forces \varphi$ because $neg(neg(\varphi)) \equiv \varphi$.
\end{proof}

\begin{lemma}[Density]\label{lem:density}
  For any $\PropL$ formula $\varphi$ and forcing condition $p$,
  there is a forcing condition $q \supseteq p$ such that
  $q \forces \varphi$ or $q \forces neg(\varphi)$.
\end{lemma}
\begin{proof}
  Assume that there is a $\PropL$ formula $\varphi$ and forcing condition $p$
  such that for all $q \supseteq p$, $q \not \forces \varphi$ and $q \not\forces neg(\varphi)$.
  Without loss of generality, suppose that $\varphi$ is a $\pSigma^p_\alpha$ formula.
  Hence, $neg(\varphi)$ is a $\pPi^p_\alpha$ formula.
  Since $p\not\forces neg(\varphi)$, it follows that for some $q \supseteq p$, $q \forces neg(neg(\varphi))$.
  But since $neg(neg(\varphi)) \equiv \varphi$, we conclude that $q \forces \varphi$. 
\end{proof}
\begin{definition}
  A bijection $g$ is called a \emph{generic} enumeration of $\A$ if for any
  $\PropLc$ formula $\varphi$, there is $p \subset g$
  such that $p \Vdash_\A \varphi$ or $p \Vdash_\A neg(\varphi)$.
\end{definition}
\begin{lemma}\label{lem:bijection}
  Every $p\in \mathcal P$ can be extended to a generic enumeration $g$.  
\end{lemma}
\begin{proof}
  As the set of $\PropLc$ formulas in normal form is countable, we can fix an
  enumeration $(\phi_n)_{n\in\omega}$ of these formulas and let $F_n=\{ q\in
  \mathcal P: q\forces \phi_n \lor q \forces \neg(\phi_n)\}$. By \cref{lem:density}
  every $F_n$ is dense in $\mathcal P$. 
  Define a sequence $p=p_0\subseteq p_1\subseteq \dots$ such that $p_{2n}\in F_n$
  and $p_{2n+1}$ contains the least natural number not in $p_{2n}$. As all the
  $F_n$ are dense, this sequence is well-defined and $g=\lim_{n\to \omega}p_n$
  is generic.
\end{proof}

\begin{theorem}[Forcing equals truth]  \label{th:forcing-equals-truth}
  Let $g$ be a generic enumeration of $\A$ and let $\G = g^{-1}(\A)$ and $D_{\G}$ be the atomic diagram of $\G$.
  Then for any $\PropLc$ formula $\varphi$,
  \[(\Nat,D_{\G}) \models \varphi \text{ iff } (\exists p \subset g)[\ p \Vdash_\A \varphi\ ].\]
\end{theorem}
\begin{proof}
  We prove this by transfinite induction on the $\PropLc$ formula $\varphi$.
  Let $\varphi$ be atomic. Then $\varphi=\top$, $\varphi=\bot$ or $\varphi
  = D(\mathbf n)$ for some $n$. The first two cases follow trivially from the
  definition. 
  For the third case, suppose that for some $p \subset g$, $p \forces D(\mathbf n)$.
  Then $p^{-1}(\A) \models \phi^{at}_n[x_i \mapsto i]$ and, thus,
  $g^{-1}(\A) \models \phi^{at}_n[x_i \mapsto i]$, or, in other words,
  $(\Nat, D_{\G}) \models D(\mathbf n)$.
  For the converse, let $(\Nat, D_{\G}) \models D(\mathbf n)$.
  By our convention that $\phi_n^{at}$ does not contain variable symbols $x_i$
  for $i>n$ we have that for $p = g\upharpoonright{n}$,
  $p^{-1}(\A) \models \phi^{at}_n[x_i \mapsto i]$ and hence $p \forces D(n)$.
  The equivalence for $\pSigma^{pc}_0$ and $\pSigma^{pc}_1$ formulas follows easily.

  Let $\varphi$ be equivalent to a $\pSigma^{pc}_\alpha$ formula, for $\alpha \geq 2$.
  Then $\varphi \equiv \vvee_{i\in W_e}(\phi_i \land \psi_i)$, where for any $i \in W_e$,
  $\phi_i$ is $\pSigma^{pc}_{\beta_i}$ and $\psi_i$ is $\pPi^{pc}_{\beta_i}$ for some $\beta_i < \alpha$.
  Suppose that for some $p \subset g$, $p \forces \varphi$, i.e., for some $i \in W_e$,
  $p \forces \phi_i$ and $p \forces \psi_i$.
  By the induction hypothesis, $(\Nat, D_{\G}) \models \phi_i$ and $(\Nat,
  D_{\G})\models \psi_i$ and hence
  $(\Nat, D_{\G}) \models \varphi$.
  For the converse, let $(\Nat, D_{\G}) \models \varphi$.
  Then $(\Nat, D_{\G}) \models \phi_i$ and $(\Nat, D_{\G}) \models \psi_i$ for some $i \in W_e$.
  By the induction hypothesis, there is some $p_0 \subset g$ such that $p_0 \forces \phi_i$
  and some $p_1 \subset g$ such that $p_1 \forces \psi_i$.
  Let $p = \max\{p_0, p_1\}$. Clearly, $p \subset g$.
  By monotonicity of forcing, $p \forces \phi_i$  and $p \forces \psi_i$.
  It follows that $p \forces \varphi$.

  Now let us consider the case when $\varphi$ is equivalent to
  a $\pPi^{pc}_\alpha$ formula, for any $\alpha \geq 0$.
  First, let $p \forces \varphi$ for some $p \subset g$.
  Towards a contradiction, assume that $(\Nat, D_{\G}) \not\models \varphi$.
  But then $(\Nat, D_{\G}) \models neg(\varphi)$, where $neg(\varphi)$ is $\pSigma^{pc}_\alpha$.
  It follows that for some $q \subset g$, $q \forces neg(\varphi)$.
  Let $r = \max\{p ,q\}$.
  By monotonicity, it follows that $r \forces \varphi$ and $r \forces neg(\varphi)$.
  A contradiction.

  Second, let $(\Nat, D_{\G}) \models \varphi$.
  Assume that for all $p \subset g$, $p \not\forces \varphi$.
  But then, since $g$ is generic, there is some $q \subset g$ such that
  $q \forces neg(\varphi)$.
  Since $neg(\varphi)$ is $\pSigma^{pc}_\alpha$, it follows that $(\Nat, D_{\G}) \models neg(\varphi)$. We reach a contradiction.
\end{proof}

\subsection*{Definability of forcing}

For any $\PropL$ formula $\varphi$, we will define a $L^p_{\omega_1\omega}$ formula $Force_\varphi(\bar{u})$ such that
\[p \Vdash_\A \varphi \text{ iff } \A \models Force_\varphi(\bar{p}). \]
\begin{definition} For any $\PropL$ formula $\varphi$ in normal form and every tuple
  $\bar u\in \omega^{<\omega}$, define the formula
  $Force_\varphi(\bar u)$ inductively as follows.
  \begin{enumerate}
    \item[(atom1)] $Force_\top(\bar u)=\exists x\; x=x$ and $Force_\bot(\bar u)=\exists
    x\exists y (x\neq y\land x=y)$
  \item[(atom2)] if $\varphi=D(n)$, then $Force_{\varphi}(\bar{u}) = \bigwedge_{i\neq j\
    \&\ i,j<n}u_i \neq u_j\ \&\ \phi^{at}_n[x_i \mapsto u_i]$
  \item[($\pSigma^p_0$)] if $\varphi = \wwedge_{i<k} \phi_i$ with $\phi_i$ atomic, then
$Force_\varphi(\bar{u})= \wwedge_{i<k} Force_{\phi_i}(\bar{u})$
\item[($\pSigma^p_1$)] if $\varphi = \vvee_{i\in I} \phi_i$ where $\phi_i\in
  \pSigma^p_0$, then $Force_{\varphi}(\bar{u}) = \vvee_{i\in
  I}Force_{\phi_i}(\ov{u})$
\item[($\pSigma^p_{>1}$)] if $\varphi= \vvee_{i\in I} (\phi_i \land \theta_i)$
  where $\phi_i\in \pSigma_{\beta_i}^p$ and $\theta_i\in \pPi^p_{\beta_i}$, $\beta_i>0$, then 
  \[Force_{\varphi}(\bar{u})= \vvee_{i\in I}(Force_{\phi_i}(\bar{u})
  \land Force_{\theta_i}(\bar{u}))\]
\item[($\pPi^p_{\alpha}$)] $Force_{\varphi}(\ov{u})=\wwedge_{k\in\omega}\forall
  x_1 \cdots \forall x_k (neg(Force_{neg(\varphi)}(\ov{u},\ov{x})))$
\end{enumerate}
\end{definition}
Notice that the definition of $Force_\varphi(\bar u)$ only depends on the formula
$\varphi$ and the size of the tuple $\bar u$. It depends neither on the tuple
$\bar u$ itself nor on the structure $\A$. 
\begin{lemma}\label{lem:defforcingcomplexity}
  For all $\alpha \geq 1$, for each $\pSigma^{pc}_\alpha$ formula $\varphi$,
  $Force_\varphi$ is a $\Sigma^{pc}_\alpha$ $\tau$-formula.
\end{lemma}
\begin{proof}
  This is proven by transfinite induction on $\alpha$. 
  \begin{itemize}
  \item
    Let $\varphi$ be a $\pSigma^p_0$ formula.
    Since $\neq$ is included in $\tau$, the formula $Force_\varphi$ is a finite
    conjunction of finite $\Sigma^p_1$ formulas and hence it is a $\Sigma^{p}_1$ formula.
  \item
    Let $\varphi$ be a $\pSigma^p_1$ formula. Then $\varphi= \vvee_{i\in I} \phi_i$,
    where all $\phi_i\in\pSigma^p_0$.
    We know that $Force_{\phi_i}(\ov{u})$ is $\Sigma^p_0$ for any tuple $\ov{u}$. Then
    \[Force_{\varphi}(\bar{u}) = \vvee_{i\in I}Force_{\phi_i}(\ov{u})\]
    is a $\Sigma^p_1$ $\tau$-formula.
  \item
    Let $\varphi$ be a $\pSigma^p_\alpha$ formula, $\alpha \geq 2$. Then
    $\varphi= \vvee_{i\in I} (\phi_i \land \theta_i)$, where
    $\phi_i\in\pSigma^p_\beta$ and $\theta_i\in\pPi^p_\beta$, for some $\beta < \alpha$.
    Without loss of generality, we may suppose that $\beta \geq 1$.
    By the induction hypothesis, we know that $Force_{\phi_i}(\ov{u})$ is $\Sigma^p_\beta$ and $Force_{\theta_i}(\ov{u})$ is $\Pi^p_\beta$.
    \[Force_{\varphi}(\bar{u}) = \vvee_{i\in I}(Force_{\phi_i}(\ov{u}) \land Force_{\theta_i}(\ov{u})).\]
    is a computable infnitary $\Sigma^p_\alpha$ $\tau$-formula.
  \item
    Let $\varphi$ be a $\pPi^p_\alpha$ formula, for any $\alpha \geq 0$.
    Then clearly, $neg(\varphi)$ is a $\pSigma^p_\alpha$ formula.
    It follows that
    $Force_{neg(\varphi)}(\ov{u})$ is a $\Sigma^p_\alpha$ $\tau$-formula for any tuple $\ov{u}$.
    Then $neg(Force_{neg(\phi_i)}(\ov{u}))$ is $\Pi^p_\alpha$ and hence
    \[Force_{\varphi}(\bar{u}) =\wwedge_{k\in\omega}\forall x_1 \cdots \forall
    x_k\; neg(Force_{neg(\varphi)}(\ov{u},\ov{x}))\]
    is $\Pi^p_\alpha$ formula if $\alpha \geq 1$ and $\Pi^p_1$ if $\alpha = 0$.
  \end{itemize}
  We still have to establish that if $\varphi$ is computable
  $\pSigma^p_\alpha$, then $Force_{\varphi}(\bar u)$ is computable
  $\Sigma^p_\alpha$. Just notice that given an index set for
  a $\pSigma^{pc}_\alpha$ formula $\varphi$ our inductive definition of
  $Force_{\varphi}(\bar u)$ describes an algorithm that produces an index set
  for a $\Sigma^{pc}_\alpha$ formula equivalent to $Force_{\varphi}(\bar u)$.
\end{proof}

\begin{lemma}[Definability of forcing]\label{lem:defforcingtruth}
  For any forcing condition $p$ and any $\PropL$ formula $\varphi$,
  $p \forces \varphi$ iff $\A \models Force_{\varphi}(\bar{p})$.
\end{lemma}
\begin{proof}
  Induction on $\varphi$.
\end{proof}

For classes of structures $\C$, $\D$ and $\K$, we say that $\C = \D$ within $\K$ if $\C \cap \K = \D \cap \K$.
We essentially repeat the proof in \cite{montalban2} to obtain an effective
version of Lopez-Escobar's theorem. 

\begin{theorem}
  \label{th:borel-mod-sentence}
  Let $\K \subseteq \Mod(\tau)$ be isomorphism invariant.
  Suppose that for $\alpha \geq 1$, $i$ is an index for a $\Pi^0_\alpha$ set $B_i$ in the effective Borel hierarchy such that
  $B_i \cap \K$ is isomorphism invariant.
  Then we can effectively find an index for a 
  $\Pi^{pc}_\alpha$ $\tau$-sentence $\varphi$ such that $B_i = \Mod(\varphi)$
  within $\K$.
\end{theorem}
\begin{proof}
  Consider a $\Pi^0_\alpha$ set $B_i$ in the effective Borel hierarchy.
  We can effectively find an index $j$ of an $\pPi^{pc}_\alpha$ formula $\phi$ such that $B_i = \Mod(\phi)$.
  By \cref{lem:defforcingcomplexity,lem:defforcingtruth} we can effectively
  find an index for the $\Pi^{pc}_\alpha$ $\tau$-sentence
  $\varphi=Force_{\phi}(\pair{})$
  where $\pair{}$ denotes the empty forcing condition.
  We will show that $B_i = \Mod(\varphi)$ within $\K$.

  First, let $\A \in B_i \cap \K$. We must show that $\A \models \varphi$.
  Towards a contradiction, assume $\A \not\models \varphi$. Then, by
  \cref{lem:defforcingtruth}, $\pair{} \not\forces \phi$.
  Since $\phi$ is $\pPi^p_\alpha$, by \cref{def:forcingrelation}, there is a forcing condition $q$
  such that $q \forces neg(\phi)$.
  Extend $q$ to a generic enumeration $g$ of $\A$ and consider $\G
  = g^{-1}(\A)$. But then by \cref{th:forcing-equals-truth}, $(\Nat,
  D_\G)\models neg(\phi)$. As $\Mod(neg(\phi))=\comp{B_i}$, $\G\in
  \comp{B_i}$. But $B_i\ni \A$ was invariant and thus, as $\A\cong \G$, $\G\in
  B_i$, a contradiction.

  Second, let $\A \in \Mod(\varphi) \cap \K$. Then $\A \models \varphi$ and hence $\pair{} \forces \phi$.
  Consider any generic copy $\G$ of $\A$.
  By Theorem~\ref{th:forcing-equals-truth}, $(\Nat,D_\G) \models \phi$ and hence $\G \in \Mod(\phi) \cap \K$.
  Since $\Mod(\phi) \cap \K$ is closed under isomorphism, $\A \in \Mod(\phi)
  \cap \K=B_i\cap \K$.



\end{proof}

Notice that \cref{th:borel-mod-sentence} also works for $\Sigma^{pc}_\alpha$ sets as $neg$ is a computable function that takes an index for a $\Sigma^{pc}_\alpha$ set and outputs an index for a $\Pi^{pc}_\alpha$ set equivalent to the negation of the original formula. Proving the theorem for $\Sigma^{pc}_\alpha$ sets directly seems to be difficult using our forcing as for $\pSigma^{pc}_\alpha$ formulas $\phi$, $\emptyset \not\forces \phi$.


\begin{theorem}\label{thm:efflopezescobar}
  Fix a computable ordinal $\alpha\geq 1$. An isomorphism invariant $B\subseteq
  Mod(\tau)$ is $\Pi^0_\alpha$ if and
  only if there is a $\Pi^{pc}_\alpha$ $\tau$-sentence $\varphi$ such that
  $B=Mod(\varphi)$.
\end{theorem}
\begin{proof}
  The left to right direction is just \cref{th:borel-mod-sentence} with
  $\K=Mod(\tau)$. The right to left direction is established by an easy
  transfinite induction.  
\end{proof}
The following corollary follows directly from the relativization of \cref{thm:efflopezescobar} after
one notices that every $\bPi^0_\alpha$ is $\Pi^0_\alpha(X)$ for some
$X\subseteq \omega$.
\begin{corollary}\label{cor:lopezescobar}
  Let $B \subseteq \Mod(\L)$ be isomorphism invariant and let $\alpha \geq 1$.
  The set $B$ is $\bPi^0_\alpha$ if and only if there is a $\Pi^p_\alpha$ sentence $\varphi$ 
  such that $B = \Mod(\varphi)$.
\end{corollary}
By the argument given after \cref{th:borel-mod-sentence} the analogs of 
\cref{thm:efflopezescobar,cor:lopezescobar} for $\bSigma^0_\alpha$ sets are
also true.

\section{A pullback theorem for computable embeddings}\label{sec:pullback}

An \emph{enumeration operator} $\Gamma$ is a c.e.\ set of pairs $\langle n, m\rangle$ where we treat $n$ as an index for $F_n$, the $n$th set in a fixed computable enumeration of the finite subsets of $\omega$. For $X\subseteq \omega$ we write $\Gamma(X)$ for the set $\{m: \langle n,m\rangle \in \Gamma \text{ and } F_n\subseteq X\}$, and for an element $f \in 2^\omega$, we let $X_f\subseteq \omega$ be such that $\chi_{X_f}=f$. Then an enumeration operator $\Gamma$ defines a function $2^\omega\to 2^\omega$ given by $f\mapsto \chi_{\Gamma(X_f)}$. In fact, it is easy to see that enumeration operators are Scott continuous (i.e., continuous functions in the Scott topology on $2^\omega$). Identifying countable models in relational vocabularies by their atomic diagrams, we may view enumeration operators as Scott continuous functions from $Mod(\tau_0)$ to $Mod(\tau_1)$ for $\tau_0,\tau_1$ relational vocabularies, and, thus, as functions between classes of structures.

Let $\K_0$ be a class of $\tau_0$-structures, and $\K_1$ be a class of
$\tau_1$-structures.


\begin{definition}
  An enumeration operator $\Gamma$ is a \emph{positive computable embedding} of $\K_0$ into $\K_1$, denoted by $\Gamma\colon \K_0 \leq_{pc} \K_1$, if $\Gamma$ satisfies the following:
  \begin{enumerate}
  	\item For any $\A\in \K_0$, $\Gamma(D(\A))$ is the atomic diagram of a structure from $\K_1$. This structure is denoted by $\Gamma(\A)$.

  	\item For any $\A,\B\in \K_0$, we have $\A\cong\B$ if and only if $\Gamma(\A) \cong \Gamma(\B)$.
  \end{enumerate}	
\end{definition}

Positive computable embeddings have the useful property of \emph{monotonicity}: If $\Gamma\colon \K_0 \leq_{pc} \K_1$ and $\A\subseteq \B$ are structures from $\K_0$, then we have $\Gamma(\A) \subseteq \Gamma(\B)$.

For total structures, i.e. structures of the form $\A = (A, R_1,\dots,R_n, \overline{R}_1,\dots,\overline{R}_n)$,
the notion of positive computable embedding coincides with the notion of computable embedding as defined in \cite{knight2007} and denoted by $\leq_c$.

Let us denote the strucrure $\A_X = (\Nat, S, X)$, where $S$ is the successor relation and $X$ is an arbitrary set.
Then it is clear that $\{\A_{K}\} \leq_c \{\A_{\overline{K}}\}$, but $\{\A_{K}\} \not\leq_{pc} \{\A_{\overline{K}}\}$, where $K$ denotes the halting set.
Moreover, $\{\A_\emptyset\} \leq_{pc} \{\A_K\}$, but $\{\A_\emptyset\} \not\leq_c \{\A_K\}$.

\begin{definition}[{\cite{knight2007}}]
  A Turing operator $\Phi=\varphi_e$ is a \emph{Turing computable embedding} of $\mathcal{K}_0$ into $\mathcal{K}_1$, denoted by $\Phi\colon \mathcal{K}_0 \leq_{tc} \mathcal{K}_1$, if $\Phi$ satisfies the following:
  \begin{enumerate}
  	\item For any $\mathcal{A}\in \mathcal{K}_0$, the function $\varphi^{D(\mathcal{A})}_e$ is the characteristic function of the atomic diagram of a structure from $\mathcal{K}_1$. This structure is denoted by $\Phi(\mathcal{A})$.

  	\item For any $\mathcal{A},\mathcal{B}\in \mathcal{K}_0$, we have $\mathcal{A}\cong\mathcal{B}$ if and only if $\Phi(\mathcal{A}) \cong \Phi(\mathcal{B})$.
  \end{enumerate}	
\end{definition}

Turing computable embeddings do in general not have the monotonicity properties
of computable embeddings. This can be used to show the following.
\begin{proposition}[Greenberg and, independently, Kalimullin; see \cite{kalimullin_computable_2018,knight2007}]
  If $\K_0 \leq_c \K_1$, then $\K_0 \leq_{tc} \K_1$. The converse is not true.
\end{proposition}


The Pullback Theorem in~\cite{vandenboom2007} connects Turing computable embeddings with definability by classical computable infinitary sentences. We obtain an analog of this result in our setting. The original pullback theorem was proved for structures whose universe can be arbitrary subsets of the natural numbers. Our version might seem a bit more restrictive at first sight, as we assume the universes of our structures to be $\omega$ and this in particular excludes finite structures. However, we could easily add unary relations for the universe of a structure and its coset to the vocabulary and then we could deal with this more general setting as well.

\begin{proposition}\label{prop:enumeration-scott}
  A positive computable embedding $\Gamma_e:\K\leq_{pc}\K'$ can be transformed into
  an effective Scott continuous function $\Gamma_a:\Mod(\tau)\to\Mod(\tau')$ with the property that $\Gamma_a(\K)=\Gamma_e(\K)$.
\end{proposition}
\begin{proof}
  Recall that for a set $A$, $\Gamma_e(A) = \{x \mid \pair{v,x} \in W_e\ \&\ D_v \subseteq A\}$.
  We may suppose that for all natural numbers $n$, $\pair{\code{\emptyset},\code{n=n}} \in W_e$.
  This ensures that all output sets can be interpreted as $\tau'$-structures with domains $\omega$.
  Since we can interpret any natural number as the code of a positive atomic sentence in $\tau'$, we only need to filter out from $W_e$ all sentences of the form $n \neq n$ and $n = m$ for different natural numbers $n$ and $m$.

  Let us have the enumeration $\{W_{e,s}\}_{s<\omega}$ of the c.e.\ set $W_e$, where $W_e = \bigcup_s W_{e,s}$, $W_{e,0} = \emptyset$ and $|W_{e,s+1}\setminus W_{e,s}| \leq 1$ for all $s$.
  We define $W_{a,s}$ inductively so that $W_{a,0} = \emptyset$ and $W_{a,s+1} = W_{a,s}$ if $\pair{v,\code{n\neq n}}$ or $\pair{v,\code{n=m}}$ belongs to $W_{e,s+1}\setminus W_{e,s}$,
  and otherwise, set $W_{a,s+1} = W_{a,s} \cup (W_{e,s+1}\setminus W_{e,s})$.
  Then clearly $\Gamma_a:\Mod(\tau) \to \Mod(\tau')$ is an effective Scott continuous mapping and by the construction of $W_a$, $\Gamma_a(\K) = \Gamma_e(\K)$.
  It is clear that we can effectively obtain the index $a$ from the index~$e$.
\end{proof}

\begin{theorem}[Pullback Theorem]\label{thm:pullback}
  Let $\K \subseteq \Mod(\tau)$ and $\K' \subseteq \Mod(\tau')$ be closed under isomorphism.
Let $\Gamma_e : \K \leq_{pc} \K'$, then for any $L^{pc}_{\omega_1\omega}$ $\tau'$-sentence $\varphi'$
we can \emph{effectively} find an $L^{pc}_{\omega_1\omega}$ $\tau$-sentence $\varphi$ such that for all $\A \in \K$,
\[\A \models \varphi \text{ iff } \Gamma_e(\A) \models \varphi'.\]
Moreover, if $\varphi'$ is $\Sigma^{pc}_\alpha$ (or $\Pi^{pc}_\alpha$), then so is $\varphi$.
\end{theorem}
\begin{proof}

  By Proposition~\ref{prop:enumeration-scott}, we can view $\Gamma_e$ as an effective Scott continuous function from $\Mod(\tau)$
  to $\Mod(\tau')$ that is isomorphism invariant on $\K$, i.e., $\A \cong \B$ if and
  only if $\Gamma_e(\A)\cong\Gamma_e(\B)$ for all $\A,\B\in \K$. Thus if $B\subseteq Mod(\tau')$ is
  invariant $\pmb\Pi^0_\alpha$, then $\Gamma_e^{-1}(B)$ is a $\pmb\Pi^0_\alpha$ subset of $Mod(\tau)$ that is invariant on $\K$.
  
  The $\pPi^p_\alpha$ formula $\phi$ describing $\Gamma_e^{-1}(B)$ can be
  effectively obtained from $\phi'$ corresponding to $B$ by
  replacing occurences of atomic formulas in $\phi'$ with the following formulas
  depending on their type:
  \[\theta_\top=\top \qquad \theta_\bot=\bot \qquad \theta_{D(\mathbf n)}=\vvee_{(n,v)\in \Gamma_e} \bigwedge_{m\in D_v}
    D(\mathbf m)\]
   An easy induction
    shows that $\phi\in \Pi^p_\alpha$ if and only if $\phi'\in \Pi^p_\alpha$
    for $\alpha>1$ and that $Mod(\phi)=\Gamma_e^{-1}(\phi')$.
  We obtained $\phi$ effectively from $\phi'$. Thus, if $B=Mod(\phi')$ is $\Pi^0_\alpha$,
  so is $\Gamma_e^{-1}(B)=Mod(\phi)$. A similar argument works for $\Sigma^p_\alpha$ formulas.
To complete the proof of the theorem use \cref{thm:efflopezescobar} to obtain $\phi'$
from $\varphi'$ and then use \cref{th:borel-mod-sentence} to obtain $\varphi$ from $\phi$ such that $Mod(\varphi)$ is equal to $\Gamma_e^{-1}(B)$ within $\K$.  \end{proof}



\section{An application}\label{sec:application}

In this section, we give a simple application of the Pullback theorem which illuminates the difference between Turing and enumeration operators.

\begin{definition}[\cite{kalimullin2018}]
  For a structure $\A$, let $\tilde{\A}$ be a structure with a new congruence relation $\sim$ such that
  \begin{enumerate}[a)]
  \item
    the $\sim$-class for each element in $\tilde{\A}$ is infinite;
  \item
    $\tilde{\A}/_{\sim} \cong \A$.
  \end{enumerate}
  For a class $\K$, let $\tilde{\K} = \{\tilde{\A} \mid \A \in \K\}$.
\end{definition}

It is not difficult to see that we always have $\tilde{\K} \equiv_{tc} \K$, $\K \leq_c \tilde{\K}$ and $\tilde{\tilde{\K}} \leq_c \tilde{\K}$.
More generally, for any two classes of structures $\K_0$ and $\K_1$, we have the following implications:
\[\K_0 \leq_c \K_1 \implies \tilde{\K}_0 \leq_c \tilde{\K}_1 \implies \K_0 \leq_{tc} \K_1.\]

We will give examples of classes of linear orderings such that the converse of these
implications does not hold. Let $\mathbf n$ denote the finite order type of
size $n$ for $n\in\omega$, and given any preorder $\A$,
let $\A^\star$ be the structure obtained by reversing the order, i.e., if
$x\leq^{\A} y$, then $y\leq^{\A^\star}x$. To see that the second implication is
strict consider the classes $\K_0 = \{{\bf 1}, {\bf 2}\}$ and
$\K_1 = \{\omega,\omega^\star\}$. It is easy to see that $\K_0 \leq_{tc} \K_1$,
but the monotonicity of enumeration operators gives us that  $\tilde\K_0 \not\leq_{c} \tilde\K_1$
because $\tilde{\bf 1}$ is a substructure of $\tilde{\bf 2}$, but
$\tilde\omega$ is \emph{not} a substructure of $\tilde\omega^\star$.

Kalimullin \cite{kalimullin2018} built a class $\K$ of finite
undirected graphs such that $\tilde{\K} \not\leq_c \K$.
This shows that the first implication is strict as we always have $\tilde{\tilde{\K}} \leq_c \tilde\K$.
Since the example in the class of graphs is somewhat artificial, Kalimullin asked the following question.
\begin{question}[Kalimullin]
  Are there natural classes of structures $\K$ such that
  \[\tilde{\K} \not\leq_c \K?\]
\end{question}

\begin{theorem}[Ganchev-Kalimullin-Vatev \cite{kalimullin2018}]
  Let $\K = \{\omega_S,\omega^\star_S\}$ be the class of strict linear orderings of types $\omega$ and $\omega^\star$,
  equipped with the successor relation. Then
  $\tilde\K \not\leq_{c} \K$.
\end{theorem}

Interestingly, it is not known whether $\{\tilde{\omega},\tilde{\omega}^\star\} \not\leq_{c} \{\omega,\omega^\star\}$,
where $\{\omega, \omega^\star\}$ is the class of \emph{strict} linear orderings of types $\omega$ and $\omega^\star$.
Nevertheless, if we let $\{\omega,\omega^\star\}$ be the class of linear orderings of order type $\omega$ and $\omega^\star$
in the vocabulary $\leq, =, \neq$, then we can apply our new version of the Pullback Theorem and show that
$\{\tilde{\omega},\tilde{\omega}^\star\} \not\leq_{pc} \{\omega,\omega^\star\}$.
For $\A$ a linear ordering, the congruence relation $x \sim y$ in $\tilde \A$ is definable by the $\Sigma^p_0$ formula $x \leq y \land y \leq x$, and, thus, we may omit $\sim$ from the vocabulary.

For a class of formulas $\C$ (e.g., $\Sigma^p_0$, $\Pi^p_1$, etc.), a structure $\A$ and a tuple $\bar{a}$ in $\A$, define
$\C\text{-}Th_\A(\bar{a}) = \{\phi(\bar{x}) \in \C \mid \A \models \phi(\bar{a})\}$. If the tuple $\bar{a}$ is empty, we will write $\C\text{-}Th_\A$.


The next proposition isolates the key property that the class $\{\tilde{\omega}, \tilde{\omega}^\star\}$ possess which
will allow us to show that $\{\tilde{\omega},\tilde{\omega}^\star\} \not\leq_{pc} \{\omega,\omega^\star\}$.

\begin{proposition}
  \label{prop:extenstion-tilde-pair}
  Let $\A$ and $\B$ be linear orderings in the vocabulary $\leq, =, \neq$.
  For any tuple $\bar{a}$ in $\tilde{\A}$ and any tuple $\bar{b}$ in $\tilde{\B}$,
  if $\Sigma^{pc}_0\text{-}Th_{\tilde{\A}}(\bar{a}) = \Sigma^{pc}_0\text{-}Th_{\tilde{\B}}(\bar{b})$,
  then $\Sigma^{pc}_1\text{-}Th_{\tilde{\A}}(\bar{a}) = \Sigma^{pc}_1\text{-}Th_{\tilde{\B}}(\bar{b})$.
\end{proposition}
\begin{proof}
  It is enough to show that $\Sigma^{pc}_1\text{-}Th_{\tilde{\A}}(\bar{a}) \subseteq \Sigma^{pc}_1\text{-}Th_{\tilde{\B}}(\bar{b})$.
  Consider a formula $\varphi(\bar{x}) = \vvee_{i\in W_e}\exists\bar{x}_i \phi_i(\bar{x},\bar{x}_i)$, where $\phi_i(\bar{x},\bar{x}_i)$ are $\Sigma^{pc}_0$.
  Suppose $\tilde{\A} \models \varphi(\bar{a})$.
  Fix $i \in W_e$ and a tuple $\bar{a}'$ in $\tilde{A}$ such that $\tilde{\A} \models \phi_i(\bar{a},\bar{a}')$.

  Since $\Sigma^{pc}_0\text{-}Th_{\tilde{\A}}(\bar{a}) = \Sigma^{pc}_0\text{-}Th_{\tilde{\B}}(\bar{b})$,
  we know that $\bar{a}$ and $\bar{b}$ are ordered in the same way.
  Without loss of generality, suppose $\bar{a}$ and $\bar{b}$ have length $n$ and $\tilde{\A} \models a_0 \leq a_1 \leq \cdots \leq a_{n-1}$ and similarly,
  $\tilde{\B} \models b_0 \leq b_1 \leq \cdots \leq b_{n-1}$. We consider two cases which show how we find the tuple $\bar{b}'$   which makes $\phi_i(\bar{b},\bar{b}')$ true in $\tilde{\B}$. In both cases we use the fact that  each $\sim$-class is infinite.
  
  Suppose that for some $a'_j$ from the tuple $\bar{a}'$, $\tilde{\A} \models a_i \leq a'_j \leq a_{i+1}$.
  The interesting case is when $a'_j$ is in a separate $\sim$-class from $a_i$ and $a_{i+1}$, but $b_{i+1}$ belongs to the successor $\sim$-class of $b_{i}$ in $\tilde{\B}$.
  Then we choose the corresponding element $b'_j$ as an element $\sim$-equivalent to $b_i$ or $b_{i+1}$.
  In this way, $\tilde{\B} \models b_{i} \leq b'_j \leq b_{i+1}$.

  Similary, suppose that for some $a'_j$ from the tuple $\bar{a}'$, $\tilde{\A} \models a_{n-1} \leq a'_j$,
  but $b_{n-1}$ belongs to the $\leq$-greatest $\sim$-class in $\tilde{\B}$.
  Then, again, we choose the corresponding $b'_j$ as an element $\sim$-equivalent to $b_{n-1}$.
\end{proof}

Notice that Proposition~\ref{prop:extenstion-tilde-pair} does not hold if we consider \emph{strict} linear oderings $\A$ and $\B$
and it does not hold for the class $\{\tilde{\omega}, \tilde{\omega}^\star\}$ if the vocabulary under consideration is $<,=,\neq$.

\begin{lemma}
  \label{lem:sigma-2-equal}
  Let $\A$ and $\B$ be linear orderings in the vocabulary $\leq, =, \neq$. Then
  \[\Sigma^{pc}_2\text{-}Th_{\tilde{\A}} = \Sigma^{pc}_2\text{-}Th_{\tilde{\B}}.\]
\end{lemma}
\begin{proof}
  It is enough to show that $\Sigma^{pc}_2\text{-}Th_{\tilde{\A}} \subseteq \Sigma^{pc}_2\text{-}Th_{\tilde{\B}}$.
  Consider a $\Sigma^{pc}_2$ sentence $\varphi$ such that 
  \[\varphi = \vvee_{i\in W_e} \exists \bar{x}_i (\phi_i(\bar{x}_i) \land neg(\psi_i(\bar{x}_i)),\]
  where $\phi_i(\bar{x}_i)$ and $\psi_i(\bar{x}_i)$ are $\Sigma^p_1$ formulas.
  Fix $i \in W_a$ and a tuple $\bar{a}$ in $\tilde{\A}$ such that
  $\tilde{\A} \models \phi_i(\bar{a}) \land neg(\psi_i(\bar{a}))$.

  Since every $\sim$-class contains infinitely many elements, it is easy to see that we can choose a tuple $\bar{b}$ in $\tilde{\B}$ such that
  $\Sigma^{pc}_0\text{-}Th_{\tilde{\A}}(\bar{a}) = \Sigma^{pc}_0\text{-}Th_{\tilde{\B}}(\bar{b})$.
  But then by Proposition~\ref{prop:extenstion-tilde-pair},
  $\Sigma^{pc}_1\text{-}Th_{\tilde{\A}}(\bar{a}) = \Sigma^{pc}_1\text{-}Th_{\tilde{\B}}(\bar{b})$.
  It follows that $\tilde{\A} \models \phi_i(\bar{b}) \land neg(\psi_i(\bar{b}))$
  and hence $\tilde{\B} \models \varphi$. 
\end{proof}

\begin{theorem}\label{thm:omegaomegastar}
  $\{\tilde{\omega},\tilde{\omega}^\star\} \not\leq_{pc} \{\omega,\omega^\star\}$.
\end{theorem}
\begin{proof}
  Consider the $\Sigma^{pc}_2$ sentences
  \[\varphi_0 = \exists x \forall y(\neg x \neq y \lor \neg y \leq x)\]
  and
  \[\varphi_1 = \exists x \forall y(\neg x \neq y \lor \neg x \leq y).\]
  It is clear that for any two copies $\A$ and $\A^\star$ of $\omega$ and $\omega^\star$, respectively,
  $\A \models \varphi_0 \land \neg \varphi_1$ and $\A^\star \models \neg \varphi_0 \land \varphi_1$.
  Notice that we don't use the formulas 
  $\exists x \forall y(x\leq y)$ and $\exists x \forall y(y \leq x)$ since they are $\Sigma^{pc}_3$ and not $\Sigma^{pc}_2$.

  If we assume that $\{\tilde{\omega},\tilde{\omega}^\star\} \leq_{pc} \{\omega,\omega^\star\}$,
  by the Pullback Theorem, there must be $\Sigma^{pc}_2$ sentences $\tilde{\varphi}_0$ and $\tilde{\varphi}_1$ such that
  for any two copies $\tilde{\A}$ and $\tilde{\A}^\star$ of $\tilde{\omega}$ and $\tilde{\omega}^\star$, respectively, 
  $\tilde{\A} \models \tilde{\varphi}_0 \land \neg \tilde{\varphi}_1$ and $\tilde{\A}^\star \models \neg \tilde{\varphi}_0 \land \tilde{\varphi}_1$.
  But this contradicts Lemma~\ref{lem:sigma-2-equal}.
\end{proof}



\printbibliography

\end{document}